\documentclass[10pt]{amsart}
\usepackage[latin1]{inputenc}
\usepackage{epsfig}
\usepackage{mathrsfs}
\usepackage{color}
\usepackage{amsthm}
\usepackage{amsmath}
\usepackage{amsfonts}
\usepackage{amssymb}
\usepackage{graphicx}
\makeatletter
\@addtoreset{equation}{section} %reinitialise le compteur equation quand section est incremente
\makeatother
\newtheorem{prop}{Proposition}[section]
\newtheorem{defi}{Definition}[section]

\newtheorem{thm}{Theorem}[section]
\newtheorem{rem}{Remark}[section]
\newtheorem{exam}{Example}[section]
\newtheorem{exams}{Examples}[section]
\newfont{\sBlackboard}{msbm10 scaled 900}

\newcommand{\ud}     {{\rm d}}
\newcommand{\mylabel}[1]{\label{#1}
            \ifx\undefined\stillediting
            \else \fbox{$#1$}\fi }

\newcommand{\EEQ}{\end{equation}}
\newcommand{\rfb}[1]{\mbox{\rm
   (\ref{#1})}\ifx\undefined\stillediting\else:\fbox{$#1$}\fi}
\newcommand{\half}   {{\frac{1}{2}}}

\newfont{\Blackboard}{msbm10 scaled 1200}

\newfont{\roma}{cmr10 scaled 1200}

\def\CC{\rm \hbox{C\kern-.56em\raise.4ex
         \hbox{$\scriptscriptstyle |$}\kern+0.5 em }}
\newcommand{\bt}{\begin{Theorem}}
\newcommand{\et}{\end{Theorem}}
\newcommand{\br}{\begin{remark}}
\newcommand{\er}{\end{remark}}
\newcommand{\bc}{\begin{Corollary}}
\newcommand{\ec}{\end{Corollary}}
\newcommand{\el}{\end{Lemma}}
\newcommand{\bd}{\begin{definition}}
\newcommand{\ed}{\end{definition}}

\newcommand{\R}  {\mathbb{R}}

\newcommand{\mm}    {{\hbox{\hskip 0.5pt}}}
\newcommand{\m}     {{\hbox{\hskip 1pt}}}

\newcommand{\bluff} {{\hbox{\raise 15pt \hbox{\mm}}}}

\newcommand{\rarrow} {{\,\rightarrow\,}}

\makeatletter
\def\section{\@startsection {section}{1}{\z@}{-3.5ex plus -1ex minus
    -.2ex}{2.3ex plus .2ex}{\large\bf}}
\makeatother
\def\ds{\displaystyle}
\newcommand{\re}{\mathrm{Re}}

\newcommand{\e}{\mathrm{e}}
\begin{document}

\thispagestyle{empty}
\title[Stabilization of fractional-evolution systems]{Stabilization of fractional-evolution systems}
\author{Ka\"{\i}s AMMARI}
\address{UR Analysis and Control of PDEs, UR 13ES64, Department of Mathematics, Faculty of Sciences of Monastir, University of Monastir, 5019 Monastir, Tunisia and LMV/UVSQ/Paris-Saclay, France} \email{kais.ammari@fsm.rnu.tn}

 \author{Fathi Hassine}
\address{UR Analysis and Control of PDEs, UR 13ES64, Department of Mathematics, Faculty of Sciences of Monastir, University of Monastir, 5019 Monastir, Tunisia} \email{fathi.hassine@fsm.rnu.tn}

\author{Luc ROBBIANO}
\address{Laboratoire de Math\'ematiques, Universit\'e de Versailles Saint-Quentin en Yvelines, 78035 Versailles, France}
\email{luc.robbiano@uvsq.fr}

\date \today

\begin{abstract}
This paper is devoted to the analysis of the problem of stabilization of fractional (in time) partial differential equations. We consider the following equation
\begin{equation*}
\ds\partial^{\alpha,\eta}_{t} u(t)=\mathcal{A}u(t)-\frac{\eta}{\Gamma (1-\alpha)}\int_{0}^{t}(t-s)^{-\alpha}\e^{-\eta(t-s)}u(s)\,\ud s,\; t > 0, 
\end{equation*}
with the initial data $u(0)=u^{0}$, where $\mathcal{A}$ is a unbounded operator in Hilbert space and $\partial_{t}^{\alpha,\eta}$ stands for the fractional derivative. We provide two main results concerning the behavior of the solutions when $t\longrightarrow+\infty$. We look first to the case $\eta>0$ where we prove that the solution of this problem is exponential stable then we consider  the case $\eta=0$ when we prove under some consideration on the resolvent that the energy of the solution goes to $0$ as $t$ goes to the infinity as $1/t^\alpha$.
\end{abstract}

\subjclass[2010]{35A01, 35A02, 35M33, 93D20}
\keywords{stabilization, abstract-wave equation, fractional-evolution}

\maketitle

\tableofcontents
%%%%%%%%%%%%%%%%%%%%%%%%%%%%%%%%%%%%%%%%%%%%%%%%%%%%%%%%%%%%%%%%%%%%%%%%%%%%%%%%%%%%%%%%%%%%%%%%%%%%%%%%%%%%%%%%%%%%%%%%%%%%%%%%%%%%%%%%%%%%%%%%%%%%%%%%%%%%%%%%%%%%%%%%%%%%SECTION%%%%%%%%%%%%%%%%%%%%%%%%%%%%%%%%%%%%%%%%%%%%%%%%%%%%%%%%%%%%%%%%%%%%%%%%%%%%%%%%%%%%%%%%%%%%%%%%%%%%%%%%%%%%%%%%%%%%%%%%%%%%%%%%%%%%%%%%%%%%%%%%%%%%%%%%%%%%%%%%%%%%%%%%%%%%%%%%%%%%%%%%%%%
\section{Introduction}
\setcounter{equation}{0}
Integer-order derivatives and integrals have clear physical interpretation and are used for describing different concepts in classical physics. For example, the position of a moving object can be represented as a function of time, the object velocity is then the first derivative of the function, the acceleration is the second derivative and so on. Fractional derivatives and integrals, being generalization of the classical derivative and integrals are expected to have even broader meaning. unfortunately, there is no such result in the literature until now.

For three centuries the theory of fractional derivatives developed  mainly as a pure theoretical field of mathematics useful only for mathematicians.  However, in the last few decades many authors pointed out that derivatives and integrals of non-integer order are very suitable for the  description of properties of various real materials, e.g. polymers. It has  been shown that new fractional-order models are more adequate than previously used integer-order models.

Fractional derivatives provide an excellent instrument for the description of memory and hereditary properties of various materials and processes. This is the main advantage of fractional derivatives in comparison with classical integer-order models, in which such effects are in fact neglected. The advantages of fractional derivatives become apparent in modeling mechanical and electrical properties of real materials, as well as in the description of rheological properties of rocks, and in many other fields.

Fractional integrals and derivatives also appear in the theory of control of dynamical systems, when the controlled system or/and the controller is described by a fractional differential equation \cite{AHR,C,QZ}.

Fractional calculus includes various extensions of the usual definition of derivative from integer to real order, including the Riemann-Liouville derivative, the Caputo derivative, the Riesz derivative, the Weyl derivative, etc. In this paper,we only consider the Caputo derivative that leads to an initial condition which is physically meaningful \cite{SKM}.

These models are relevant, in particular, in the context of spatially disordered systems, porous media, fractal media, turbulent fluids and plasmas, biological media with traps, binding sites or macro-molecular crowding, stock price movements, etc. We refer the readers to \cite{BH,BG,MK} and the rich references therein for the motivation and description of the model. On the other hand, we refer to \cite{CCV,CV,CSS,CS,CKRZ,KNV} and the rich references therein for mathematical analysis of these models.

Let $H$ be a Hilbert space equipped with the norm $\|\,.\,\|_H$, and let $\mathcal{A}:\mathcal{D}(\mathcal{A})\subset H \longrightarrow H$ be a closed and densely defined operator on $H$. We consider the following Cauchy problem described by the mean of the fractional derivative as follow:
\begin{equation}\label{IFE1}
\left\{
\begin{array}{ll}
\ds\partial^{\alpha,\eta}_{t} u(t)=\mathcal{A}u(t)-\frac{\eta}{\Gamma (1-\alpha)}\int_{0}^{t}(t-s)^{-\alpha}\e^{-\eta(t-s)}u(s)\,\ud s,\; t > 0, 
\\
u(0)=u^{0},
\end{array}
\right.
\end{equation}
where $\partial_{t}^{\alpha,\eta}$ denoted the fractional derivative defined by
\begin{equation}\label{IFE2}
\partial_{t}^{\alpha,\eta}v (t)=\frac{1}{\Gamma (1-\alpha)}\,\int_{0}^{t} (t-s)^{-\alpha}\e^{-\eta(t-s)}\,v^\prime(s) \,\ud s,\;0<\alpha<1,\;\eta\geq 0.
\end{equation}
%There are many definitions for fractional derivatives \cite{das}, among which Riemann-Liouville definition and Caputo definitions are most widely used \cite{MJBSR}. 
The main result of this paper concerns the precise asymptotic behavior of the solutions of \eqref{IFE1}. In \cite{C}, the author consider a fractional integro-differential equation equivalent in some sense to system \eqref{IFE1} but for the case $1<\alpha<2$ and $\eta=0$ where he shows how the asymptotic behavior of the continuous solution depends on some parameter $\omega$ where it's assumed that $\mathcal{A}$ is a sectoral operator with  a sector depends on $\omega$. Precisely, if $\omega\geq0$, then the continuous solutions are bounded by an exponential of type $\e^{\omega^{\frac{1}{\alpha}}t}$ and if $\omega<0$, then the solutions show a merely algebraic decay of order $o\left(\frac{1}{\omega t^{\alpha}}\right)$. In this work we prove under some consideration on the resolvent behavior (weaker then the one considered in \cite{C}) that the second kind of behavior still true if $\eta=0$ and one shows also that the energy is exponentially stable if $\eta>0$.

This paper is organized as follows: In section \ref{WPFE} we prove the well-posedness of system \eqref{IFE1} and gives an exponential stability result in the case $\eta>0$. In section \ref{PSFE} we prove that the energy of system \eqref{IFE1}  is polynomially stable for the case $\eta=0$ while in section \ref{WEFE} we consider an integro-differential equation where we prove the well-posedness of the equation and a polynomial decay rate of the energy.
%%%%%%%%%%%%%%%%%%%%%%%%%%%%%%%%%%%%%%%%%%%%%%%%%%%%%%%%%%%%%%%%%%%%%%%%%%%%%%%%%%%%%%%%%%%%%%%%%%%%%%%%%%%%%%%%%%%%%%%%%%%%%%%%%%%%%%%%%%%%%%%%%%%%%%%%%%%%%%%%%%%%%%%%%%%%%%%%%%%%%%%%%%%%%%%%%%%%%%%%%%%%%%%%%%%%%%%%%%%%%%%%%%%%%%%%%%%%%%%%%%%%%%%%%%SECTION%%%%%%%%%%%%%%%%%%%%%%%%%%%%%%%%%%%%%%%%%%%%%%%%%%%%%%%%%%%%%%%%%%%%%%%%%%%%%%%%%%%%%%%%%%%%%%%%%%%%%%%%%%%%%%%%%%%%%%%%%%%%%%%%%%%%%%%%%%%%%%%%%%%%%%%%%%%%%%%%%%%%%%%%%%%%%%%%%%%%%%%%%%%%%%%%%%%%%%%%%%%%%%%%%%%%%%%%%%%%%%%%%%%%%%%%%%%%%%%%%%%%%%%%%%%%
\section{Well-posedness and exponential stabilization}\label{WPFE}
%We define the unbounded operator $\mathcal{A}$ by
%$$
%\mathcal{A}u=-iAw-BB^{*}u
%$$
%with domain
%$$
%\mathcal{D}(\mathcal{A})=\left\{u\in H: iAu+BB^{*}u\in H\right\},
%$$
%then system \eqref{IFE1}-\eqref{IFE2} can be written as the following Cauchy problem
%\begin{equation}\label{WPFE1}
%\left\{\begin{array}{ll}
%\ds\partial_{t}^{\alpha,\eta}u(t)+\frac{\eta}{\Gamma (1-\alpha)}\int_{0}^{t}(t-s)^{-\alpha}\e^{-\eta(t-s)}u(s)\,\ud s=\mathcal{A}u(t),&t>0,
%\\
%u(0)=u^{0}.&
%\end{array}\right.
%\end{equation}
We define the convolution product of $a$ and $u$ by 
$$
a*u(t)=\int_{0}^{t}a(t-s)u(s)\,\ud s,\quad\forall\,t\in\R_{+},\;a\in L_{\mathrm{loc}}^{1}(\R_{+}),\;u\in L_{\mathrm{loc}}^{p}(\R_{+},H),\, p\in[1,+\infty[,
$$
and for $\beta>0$ the functions $g_{\beta}$ are given by
\begin{equation*}
g_{\beta}(t)=\left\{\begin{array}{ll}
\ds\frac{1}{\Gamma(\beta)}t^{\beta-1}&t>0,
\\
0&t\leq0.
\end{array}\right.
\end{equation*}
Noting that these functions satisfy the semigroup property, namely
\begin{equation}\label{WPFE12}
g_{\beta}*g_{\gamma}(t)=g_{\beta+\gamma}(t),\quad\forall\,t>0,\;\beta,\gamma>0.
\end{equation}
The Riemann-Liouville fractional integral of order $0<\alpha<1$ is defined as follow
$$
J^{\alpha}u(t)=g_{\alpha}*u(t),\;\forall\,u\in L_{\mathrm{loc}}^{1}(\R_{+};H),\;t>0.
$$
Using \eqref{WPFE12} we follow that $J^{\alpha}$ verifying the semigroup property,
\begin{equation}\label{WPFE16}
J^{\beta}J^{\gamma}u(t)=J^{\beta+\gamma}u(t),\quad\forall\,t>0,\;\beta,\gamma>0.
\end{equation}
For every $u\in L_{\mathrm{loc}}^{1}(\R_{+};H)$ such that $g_{1-\alpha}*u\in W_{\mathrm{loc}}^{1,1}(\R_{+};H)$, the Riemann-Liouville fractional derivative of order $\alpha$ is defined by
$$
D_{t}^{\alpha}u(t)=(g_{1-\alpha}*u)^{'}(t)=(J^{1-\alpha}u)^{'}(t),\;\forall\,t>0.
$$
The operator $D_{t}^{\alpha}$ is the left inverse and right invertible of $J^{\alpha}$ (see \cite[Theorem 1.5]{bajlekova}) more precisely, we have
\begin{equation}\label{WPFE14}
D_{t}^{\alpha}J^{\alpha}u(t)=u(t),\quad\forall\,u\in L_{\mathrm{loc}}^{1}(\R_{+};H),\;t>0,
\end{equation}
and for every $u\in L_{\mathrm{loc}}^{1}(\R_{+};H)$ such that $g_{1-\alpha}*u\in W_{\mathrm{loc}}^{1,1}(\R_{+};H)$,
\begin{equation}\label{WPFE15}
J^{\alpha}D_{t}^{\alpha}u(t)=u(t)\,\quad\forall\,t>0.
\end{equation}
We recall that the Caputo fractional derivative of order $\alpha>0$ when $\eta=0$ is defined by
$$
\mathrm{\textbf{D}}_{t}^{\alpha}u(t)=J^{1-\alpha}u'(t),\quad\forall\,u\in W_{\mathrm{loc}}^{1,1}(\R_{+};H),\;t>0,
$$
then by integration by parts we follow that when $u\in W_{\mathrm{loc}}^{1,1}(\R_{+};H)$, we have
$$
\mathrm{\textbf{D}}_{t}^{\alpha}u(t)=D_{t}^{\alpha}(u-u(0))(t),\quad\forall\,t>0.
$$
Hence, the Caputo derivative $\mathrm{\textbf{D}}_{t}^{\alpha}$ is a left inverse of $J^{\alpha}$ but in general it is not a right inverse, namely using \eqref{WPFE14} and \eqref{WPFE15} we have 
\begin{equation}\label{WPFE7}
\mathrm{\textbf{D}}_{t}^{\alpha}J^{\alpha}u(t)=u(t),\quad\forall\,u\in L_{\mathrm{loc}}^{1}(\R_{+}),\;t>0,
\end{equation}
and
\begin{equation}\label{WPFE8}
J^{\alpha}\mathrm{\textbf{D}}_{t}^{\alpha}u(t)=u(t)-u(0),\quad\forall\,u\in\mathcal{C}(\R_{+};H),\,g_{1-\alpha}*(u-u(0))\in W_{\mathrm{loc}}^{1,1}(\R_{+};H),\;t>0.
\end{equation}
By setting $v(t)=u(t)\e^{\eta t}$, \eqref{IFE1} is equivalent to the following problem
\begin{equation}\label{WPFE6}
\left\{\begin{array}{ll}
\mathrm{\textbf{D}}_{t}^{\alpha}v(t)=\mathcal{A}v(t),&t>0,
\\
v(0)=u^{0}.&
\end{array}\right.
\end{equation}
Applying $J_{\alpha}$ in both sides of the first line of \eqref{WPFE6} and using \eqref{WPFE7} and \eqref{WPFE8}, we conclude that when $u\in\mathcal{C}(\R_{+};H)$ satisfying $g_{1-\alpha}*(v-u^{0})\in W_{\mathrm{loc}}^{1,1}(\R_{+};H)$ \eqref{WPFE6} is equivalent to the following integral differential equation
\begin{equation}\label{WPFE9}
u(t)=\e^{-\eta t}u^{0}+\e^{-\eta t}(g_{\alpha}*(\e^{\eta.}\mathcal{A}u))(t),\;\forall\,t\geq 0.
\end{equation} 

The well-posedness of a system such as \eqref{IFE1} is related to the notion of what is called solution operator defined as follow:
\begin{defi}
A family $(S_{\alpha,\eta}(t))_{t\geq 0}\in\mathcal{L}(H)$ (denoted simply by $(S_{\alpha}(t))_{t\geq 0}$ if $\eta=0$) is called a solution operator (or a resolvent) for \eqref{IFE1} or for \eqref{WPFE9} if the following conditions are satisfied
\begin{itemize}
	\item[(a)] $S_{\alpha,\eta}(t)$ is strongly continuous for $t\geq0$ and $S_{\alpha,\eta}(0)=I$.
	\item[(b)] $S_{\alpha,\eta}(t)(\mathcal{D}(\mathcal{A}))\subset\mathcal{D}(\mathcal{A})$ and $\mathcal{A}S_{\alpha,\eta}(t)x=S_{\alpha,\eta}(t)\mathcal{A}x$ for all $x\in\mathcal{D}(\mathcal{A})$ and $t\geq0$.
	\item[(c)] The resolvent equation holds 
	$$
	S_{\alpha,\eta}(t)x=\e^{-\eta t}x+\e^{-\eta t}(g_{\alpha}*(\e^{\eta.}\mathcal{A}S_{\alpha,\eta}(.)x))(t),\quad \forall\,x\in\mathcal{D}(\mathcal{A}),\;t\geq 0.
	$$
\end{itemize}
\end{defi}
\begin{defi}
A function $u\in\mathcal{C}(\R_{+},H)$ is called a strong solution of \eqref{WPFE9} if $u\in\mathcal{C}(\R_{+},\mathcal{D}(\mathcal{A}))$ and \eqref{WPFE9} holds on $\R_{+}$.
\end{defi}
\begin{defi}
The problem \eqref{WPFE9} is called well-posed if for any $u^{0}\in\mathcal{D}(\mathcal{A})$, there is a unique strong solution $u(t,u^{0})$ of \eqref{WPFE9}, and when $u_{n}^{0}\in\mathcal{D}(\mathcal{A})$ such that $u_{n}^{0}\,\longrightarrow\,0$ as $n\nearrow+\infty$ then $u(t,u_{n}^{0})\,\longrightarrow\,0$ as $n\nearrow+\infty$ in $H$, uniformly on compact intervals.
\end{defi}
\begin{defi}
A function $u\in\mathcal{C}(\R_{+},H)$ is called a strong solution of \eqref{IFE1} if $u\in\mathcal{C}(\R_{+},\mathcal{D}(\mathcal{A}))$, $g_{1-\alpha}*((\e^{\eta.}u)-u^{0})\in\mathcal{C}^{1}(\R_{+},H)$ and \eqref{IFE1} holds on $\R_{+}$.
\end{defi}
\begin{defi}\label{WPFE17}
The problem \eqref{IFE1} is called well-posed if for any $u^{0}\in\mathcal{D}(\mathcal{A})$, there is a unique strong solution $u(t,u^{0})$ of \eqref{IFE1}, and when $u_{n}^{0}\in\mathcal{D}(\mathcal{A})$ such that $u_{n}^{0}\,\longrightarrow\,0$ as $n\nearrow+\infty$ then $u(t,u_{n}^{0})\,\longrightarrow\,0$ as $n\nearrow+\infty$ in $H$, uniformly on compact intervals.
\end{defi}
\begin{prop}\label{WPFE5}
\begin{enumerate}
	\item[1/] Equation \eqref{WPFE9} is well-posed if and only if it admits a resolvent $S_{\alpha,\eta}(t)$. If this is the case we have in addition $(g_{\alpha}*(\e^{\eta.}S_{\alpha,\eta}(.)x))(t)\in\mathcal{D}(\mathcal{A})$ for all $x\in H$ and $t\geq 0$ and we have
\begin{equation}\label{WPFE11}
S_{\alpha,\eta}(t)x=\e^{-\eta t}x+\e^{-\eta t}\mathcal{A}(g_{\alpha}*(\e^{\eta.}S_{\alpha,\eta}(.)x))(t),\quad \forall\,x\in H,\;t\geq 0.
\end{equation}
	\item[2/] System \eqref{IFE1} is well-posed if and only if \eqref{WPFE9} is well-posed.
\end{enumerate}
\end{prop}
\begin{proof}
\begin{enumerate}
	\item[1/] Problem \eqref{WPFE9} can be written as follows:
	\begin{equation}\label{WPFE10}
	v(t)=u^{0}+g_{\alpha}*\mathcal{A}v(t),\quad t\geq 0,
	\end{equation}
	where we denoted by $v(t)=\e^{\eta t}u(t)$. Following to \cite[Proposition 1.1]{pruss2}, \eqref{WPFE10} is well-posed if and only if it admits a solution operator $\tilde{S}_{\alpha}(t)$. The result follow easily by observing that $S_{\alpha,\eta}(t)=\tilde{S}_{\alpha}(t)\e^{-\eta t}$ is a solution operator of \eqref{WPFE9} and then \eqref{WPFE11} holds. 
	\item[2/] The first implication follows easily from the definitions. Now suppose that \eqref{WPFE9} is well-posed and let $u$ its solution and let $v$ as above solution of \eqref{WPFE10}. To prove the result we only have to prove that $g_{1-\alpha}*(v-u^{0})\in\mathcal{C}^{1}(\R_{+},H)$. Since $v \in\mathcal{C}(\R_{+},\mathcal{D}(\mathcal{A}))$ then the convolution of \eqref{WPFE10} with $g_{1-\alpha}$ gives
	$$
	g_{1-\alpha}*(v-u^{0})(t)=\mathcal{A}\int_{0}^{t}v(s)\,\ud s,\quad t\geq 0,
	$$
	where we have used \eqref{WPFE12}. Then it is easy to show that $g_{1-\alpha}*(v-u^{0})\in\mathcal{C}^{1}(\R_{+},H)$.
\end{enumerate}
This completes the proof.
\end{proof}
\begin{defi}
The solution operator $S_{\alpha}(t)$ ($\eta=0$) is called exponentially bounded if there are $M\geq1$ and $\omega\geq 0$ such that
\begin{equation}\label{WPFE2}
\|S_{\alpha}(t)\|_{\mathcal{L}(H)}\leq M\e^{\omega t}.
\end{equation}
The operator $\mathcal{A}$ is said to belong to $\mathscr{C}^{\alpha}(M,\omega)$ if the problem \eqref{IFE1} has a solution operator $S_{\alpha}(t)$ satisfying \eqref{WPFE2}. Denote $\mathscr{C}^{\alpha}(\omega)=\cup\left\{\mathscr{C}^{\alpha}(M,\omega):\;M\geq1\right\}$ and $\mathscr{C}^{\alpha}=\cup\left\{\mathscr{C}^{\alpha}(\omega):\;\omega\geq0\right\}$.
\end{defi}
For $\theta\in[0,\pi)$ we denote by
$$
\Sigma_{\theta}=\left\{z\in\mathbb{C}^{*}:\;|\arg(z)|<\theta\right\}.
$$
\begin{defi}
A solution operator $S_{\alpha}(t)$ of \eqref{WPFE6} is called analytic if $S_{\alpha}(t)$ admits an analytic extension to a sector $\Sigma_{\theta_{0}}$ for some $\ds\theta_{0}\in(0,\frac{\pi}{2}]$. An analytic solution operator is said to be of analytic type $(\theta_{0},\omega_{0})$ if for each $\theta<\theta_{0}$ and $\omega>\omega_{0}$ there is $M=M(\theta,\omega)$ such that
\begin{equation*}
\|S_{\alpha}(t)\|_{\mathcal{L}(H)}\leq M\e^{\omega \re(t)},\;\forall\,t\in\Sigma_{\theta}.
\end{equation*}
Where $\re(\lambda)$ stands for the real part of $\lambda$. where $\re(\lambda)$ stands for the real part of $\lambda$. The set of all operators $\mathcal{A}\in\mathscr{C}^{\alpha}$, generating analytic solution operator $S_{\alpha,\eta}(t)$ of type $(\theta_{0},\omega_{0})$ is denoted by $\mathscr{A}^{\alpha}(\theta_{0},\omega_{0})$. In addition, denote $\ds\mathscr{A}^{\alpha}(\theta_{0})=\cup\left\{\mathscr{A}^{\alpha}(\theta_{0},\omega_{0}):\,\omega_{0}\in\R_{+}\right\}$ and $\ds\mathscr{A}^{\alpha}=\cup\left\{\mathscr{A}^{\alpha}(\theta_{0}):\,\theta_{0}\in]0,\frac{\pi}{2}]\right\}$.
\end{defi}
\begin{prop}\cite[Corollary 2.17]{bajlekova}\label{WPFE4}
Suppose that $\left\{\lambda:\,\re(\lambda)>0\right\}\subset\rho(\mathcal{A})$ and for some $C>0$ we have
\begin{equation}\label{WPFE3}
\|(\lambda I-\mathcal{A})^{-1}\|_{\mathcal{L}(H)}\leq\frac{C}{\re(\lambda)},\quad\forall\,\lambda\in\rho(\mathcal{A}),\;\re(\lambda)>0.
\end{equation}
Then for any $\alpha\in(0,1)$, $\ds\mathcal{A}\in\mathscr{A}^{\alpha}\left(\min\left\{\left(\frac{1}{\alpha}-1\right),1\right\}\frac{\pi}{2},0\right)$.
\end{prop}
\begin{thm}\label{WPEFE17}
Suppose that $\mathcal{A}$ is a m-dissipative operator on the Hilbert space $H$, then $\mathcal{A}$ generates a solution operator $S_{\alpha,\eta}(t)$ and system \eqref{IFE1} is well-posed. In particular, when $\eta>0$ and $\alpha \in (0,1)$ system \eqref{IFE1} is exponentially stable and for some $M>0$ we have
\begin{equation}\label{WPFE13}
\|S_{\alpha,\eta}(t)u^{0}\|_{H}\leq M\e^{-\eta t}\|u^{0}\|_{H},\;\forall\,u^{0}\in\mathcal{D}(\mathcal{A}),\;t\geq 0.
\end{equation}
\end{thm}
\begin{proof}
%Let $u\in H$ then we have
%$$
%\re\langle\mathcal{A}u,u\rangle=-\|B^{*}u\|_{U}^{2},
%$$
%where $\langle.,.\rangle$ stands for the scalar product in the Hilbert space $H$. This shows that $\mathcal{A}$ is dissipative, then with the assumption made in this theorem 
Since the operator $\mathcal{A}$ is m-dissipative then by \cite[Proposition 3.1.9]{tucsnakweinss} property \eqref{WPFE3} holds. According to Proposition \ref{WPFE4}, $\mathcal{A}$ is a generator of solution operator $S_{\alpha}(t)$ of \eqref{WPFE6}, therefore Proposition \ref{WPFE5} leads to the well-posedness of the problem \eqref{WPFE6} and consequently the well-posedness of \eqref{IFE1} since $v(t)=u(t)\e^{\eta t}$. Besides, for some constant $M>0$ we have $v(t)=S_{\alpha}(t)u^{0}$ satisfies
$$
\|S_{\alpha}(t)u^{0}\|_{H}\leq M\|u^{0}\|_{H},\;\forall\,t\geq 0.
$$
Therefore, the solution operator since $u(t)=v(t)\e^{-\eta t}$ and \eqref{WPFE13} holds. This completes the proof.
\end{proof}

\begin{exams}
As examples we consider the following systems:

\begin{equation}\label{EXA1}
\left\{
\begin{array}{ll}
\ds\partial^{\alpha,\eta}_{t} u(t,x)+ (i \Delta + a(x))u(t,x)+\frac{\eta}{\Gamma (1-\alpha)}\int_{0}^{t}(t-s)^{-\alpha}\e^{-\eta(t-s)}u(s,x)\,\ud s=0,\; t > 0, x \in \Omega,\\
u = 0, \, (0,+\infty) \times \partial \Omega,
\\
u(0,x)=u^{0}(x), \, x \in \Omega,
\end{array}
\right.
\end{equation}

and 

\begin{equation}\label{EXA2}
\left\{
\begin{array}{ll}
\ds\partial^{\alpha,\eta}_{t} u(t,x)+ (i \Delta + i) u(t,x) +\frac{\eta}{\Gamma (1-\alpha)}\int_{0}^{t}(t-s)^{-\alpha}\e^{-\eta(t-s)}u(s,x)\,\ud s=0,\; t > 0, x \in \Omega, \\
\partial_\nu u  = i b(x) \, u, \, (0,+\infty) \times \partial \Omega,
\\
u(0,x)=u^{0}(x), \, x \in \Omega,
\end{array}
\right.
\end{equation}
where $\Omega$ is a smooth bounded open domain of $\R^n,$ $\partial_\nu = \nu . \nabla$ is the derivative along $\nu$, the unit normal vector pointing outward of $\Omega$
and $a \in L^\infty (\Omega), b \in L^\infty(\partial \Omega)$ are non-identically  zero and non-negative functions. 

\medskip

By a direct implication of Theorem \ref{WPEFE17} we obtain for $\eta>0$ exponential stability results for \rfb{EXA1} and \rfb{EXA2} without any geometric conditions (see \cite{blr} for example) on the supports of $a$ and $b$.

\end{exams}
%%%%%%%%%%%%%%%%%%%%%%%%%%%%%%%%%%%%%%%%%%%%%%%%%%%%%%%%%%%%%%%%%%%%%%%%%%%%%%%%%%%%%%%%%%%%%%%%%%%%%%%%%%%%%%%%%%%%%%%%%%%%%%%%%%%%%%%%%%%%%%%%%%%%%%%%%%%%%%%%%%%%%%%%%%%%%%%%%%%%%%%%%%%%%%%%%%%%%%%%%%%%%%%%%%%%%%%%%%%%%%%%%%%%%%%%%%%%%%%%%%%%%%%%%%SECTION%%%%%%%%%%%%%%%%%%%%%%%%%%%%%%%%%%%%%%%%%%%%%%%%%%%%%%%%%%%%%%%%%%%%%%%%%%%%%%%%%%%%%%%%%%%%%%%%%%%%%%%%%%%%%%%%%%%%%%%%%%%%%%%%%%%%%%%%%%%%%%%%%%%%%%%%%%%%%%%%%%%%%%%%%%%%%%%%%%%%%%%%%%%%%%%%%%%%%%%%%%%%%%%%%%%%%%%%%%%%%%%%%%%%%%%%%%%%%%%%%%%%%%%%%%%%
\section{Polynomial stabilization}\label{PSFE}
The aim of this section is to establish a polynomial stabilization result of the system \eqref{IFE1} only for the case $\eta=0$. For this purpose we introduce first some properties of the Mittag-Leffler function (\cite[Chapter XVIII]{EMOT} and \cite[chapter 1]{podlubny}) $E_{\alpha,\beta}$ defined by 
$$
E_{\alpha,\beta}(z)=\sum_{n=0}^{+\infty}\frac{z^{n}}{\Gamma(\alpha n+\beta)}=\frac{1}{2i\pi}\int_{C}\frac{\mu^{\alpha-\beta}\e^{\mu}}{\mu^{\alpha}-z}\,\ud\mu,\quad\forall\,z\in\mathbb{C},\;\alpha,\beta>0,
$$
where $C$ is a contour which starts and ends at $-\infty$ and encircles the disc $D=\{\mu\in\mathbb{C}:\;|\mu|\leq|z|^{\frac{1}{\alpha}}\}$ counter-clockwise. For short, we denote $E_{\alpha}(z)=E_{\alpha,1}(z)$. The first property claims (see \cite[Theorem 1.6]{podlubny}) that for every $\beta>0$ and $0<\alpha<2$, there exists a constant $c>0$, such that
\begin{equation}\label{PSFE1}
|E_{\alpha,\beta}(-t)|\leq\frac{c}{1+t},\quad\forall\,t>0.
\end{equation}
Consider also the function of Wright type $\phi_{\gamma}$ (see \cite{GLM,Mainardi,wright}) given by
$$
\Phi_{\gamma}(z)=\sum_{n=0}^{+\infty}\frac{(-z)^{n}}{n!\Gamma(1-\gamma(n+1))}=\frac{1}{2i\pi}\int_{C'}\mu^{\gamma-1}\e^{\mu-z\mu^{\gamma}}\,\ud\mu,\quad 0<\gamma<1,
$$
where $C'$ is a contour which starts and ends at $-\infty$ and encircles the origin once counter-clockwise. The relationship between the Mittag-Leffler function $E_{\gamma}$ and the function of Wright type $\Phi_{\gamma}$ is given by
\begin{equation}\label{PSFE2}
E_{\gamma}(z)=\int_{0}^{+\infty}\Phi_{\gamma}(t)\,\e^{zt}\,\ud t,\quad \forall\,z\in\mathbb{C},\;0<\gamma<1.
\end{equation}
That is, $E_{\gamma}(-z)$ is the Laplace transform of $\Phi_{\gamma}$ in the whole complex plane. Therefore, $\Phi_{\gamma}$ is a probability density function,
\begin{equation}\label{PSFE3}
\Phi_{\gamma}(t)\geq 0,\; \forall\,t>0;\quad\text{and}\quad \int_{0}^{+\infty}\Phi_{\gamma}(t)\,\ud t=1.
\end{equation}
One of the main ingredients of this section is the following proposition.
\begin{prop}\cite[Theorem 3.1]{bajlekova}\label{PSFE4}
Let $0<\alpha<\beta\leq 2$, $\ds\gamma=\frac{\alpha}{\beta}$ and $\omega\geq 0$. If $\mathcal{A}\in\mathscr{C}^{\beta}(\omega)$ then $\ds\mathcal{A}\in\mathscr{C}^{\alpha}(\omega^{\frac{1}{\gamma}})$ and the following representation holds
\begin{equation}\label{PSFE11}
S_{\alpha}(t)=\int_{0}^{+\infty}\varphi_{t,\gamma}(s)S_{\beta}(s)\,\ud s,\quad\forall\,t>0,
\end{equation}
where $\varphi_{t,\gamma}(s)=t^{-\gamma}\Phi_{\gamma}(st^{-\gamma})$. The identity \eqref{PSFE11} holds in the strong sense.
\end{prop}
The main result of this section is given by the following theorem.
\begin{thm}\label{PSFE10}
We suppose that $\mathcal{A}$ generates a $C_{0}$-semigroup on the Hilbert space $H$ such that the following properties hold,
\begin{equation}\label{PSFE6}
\re (\lambda) < 0, \, \forall \, \lambda \in \sigma(\mathcal{A}),\quad\text{and}\quad\sup_{\re (\lambda) \geq 0}\|(\lambda I-\mathcal{A})^{-1}\|_{\mathcal{L}(H)}<+\infty.
\end{equation}
Then \eqref{IFE1} admits a solution operator $S_{\alpha}(t)$ such that there exists $c>0$
\begin{equation}\label{PSFE7}
\|S_{\alpha}(t)\|_{\mathcal{L}(H)}\leq\frac{c}{1+t^{\alpha}},\quad\forall\,t\geq0.
\end{equation}
\end{thm}
\begin{proof}
With the assumption made on the theorem we can apply Proposition \ref{PSFE4} with $\beta=1$ and of course $\gamma=\alpha$ then \eqref{IFE1} admits a solution operator $S_{\alpha}(t)$ given by the following formula,
\begin{equation}\label{PSFE5}
S_{\alpha}(t)=\int_{0}^{+\infty}\varphi_{t,\alpha}(s)S(s)\,\ud s,\quad\forall\,t>0,
\end{equation}
where we denoted  $S_{1}(t)$ simply by $S(t)$ which is the $C_{0}$-semigroup generated by $\mathcal{A}$. Thanks to the assumptions \eqref{PSFE6}, then according to \cite{huang,pruss} the uniform stabilization holds, that is there exist $\omega_{0}>0$ and $K>0$ such that
\begin{equation}\label{PSFE8}
\|S(t)\|_{\mathcal{L}(H)}\leq K\e^{-\omega_{0}t}\quad\forall\,t>0.
\end{equation}
Performing a change of variable on \eqref{PSFE5} and using \eqref{PSFE3} and \eqref{PSFE8} we obtain
\begin{equation}\label{PSFE9}
\|S_{\alpha}(t)\|_{\mathcal{L}(H)}=\left\|\int_{0}^{+\infty}\Phi_{\alpha}(s)S(st^{\alpha})\,\ud s\right\|_{\mathcal{L}(H)}\leq\int_{0}^{+\infty}\Phi_{\alpha}(s)\e^{-\omega_{0}st^{\alpha}}\,\ud s\quad\forall\,t>0.
\end{equation}
Following to \eqref{PSFE2} and \eqref{PSFE9} we find 
$$
\|S_{\alpha}(t)\|_{\mathcal{L}(H)}\leq E_{\alpha}(-\omega_{0}t^{\alpha})\quad\forall\,t>0.
$$
Estimate \eqref{PSFE7} follows now from \eqref{PSFE1} and this completes the proof.
\end{proof}

 \begin{exams}
As examples we consider here the same systems as above but with $\eta =0$:

\begin{equation}\label{EXA1p}
\left\{
\begin{array}{ll}
\ds\partial^{\alpha,\eta}_{t} u(t,x)+ (i \Delta + a(x))u(t,x) =0,\; t > 0, x \in \Omega,\\
u = 0, \, (0,+\infty) \times \partial \Omega,
\\
u(0,x)=u^{0}(x), \, x\in \Omega,
\end{array}
\right.
\end{equation}

and 

\begin{equation}\label{EXA2p}
\left\{
\begin{array}{ll}
\ds\partial^{\alpha,\eta}_{t} u(t,x)+ (i \Delta + i) u(t,x) =0,\; t > 0, x \in \Omega, \\
\partial_\nu u  = i b(x) \, u, \, (0,+\infty) \times \partial \Omega,
\\
u(0,x)=u^{0}(x), \, x \in \Omega,
\end{array}
\right.
\end{equation}
where $\Omega$ is a smooth bounded open domain of $\R^n, $ 
$\partial_\nu = \nu . \nabla$ is the derivative along $\nu$, the unit normal vector pointing outward of $\Omega$
and $a \in L^\infty (\Omega), b \in L^\infty(\partial \Omega)$ are non-identically zero and non-negative functions. 

\medskip

By a direct implication of Theorem \ref{PSFE10} we obtain polynomial stability results for \rfb{EXA1p} and \rfb{EXA2p} under geometric conditions G.C.C. (see, respectively, \cite{leb} and \cite{blr}, for example) on the supports of $a$ and $b$.

\end{exams}
%%%%%%%%%%%%%%%%%%%%%%%%%%%%%%%%%%%%%%%%%%%%%%%%%%%%%%%%%%%%%%%%%%%%%%%%%%%%%%%%%%%%%%%%%%%%%%%%%%%%%%%%%%%%%%%%%%%%%%%%%%%%%%%%%%%%%%%%%%%%%%%%%%%%%%%%%%%%%%%%%%%%%%%%%%%%%%%%%%%%%%%%%%%%%%%%%%%%%%%%%%%%%%%%%%%%%%%%%%%%%%%%%%%%%%%%%%%%%%%%%%%%%%%%%%SECTION%%%%%%%%%%%%%%%%%%%%%%%%%%%%%%%%%%%%%%%%%%%%%%%%%%%%%%%%%%%%%%%%%%%%%%%%%%%%%%%%%%%%%%%%%%%%%%%%%%%%%%%%%%%%%%%%%%%%%%%%%%%%%%%%%%%%%%%%%%%%%%%%%%%%%%%%%%%%%%%%%%%%%%%%%%%%%%%%%%%%%%%%%%%%%%%%%%%%%%%%%%%%%%%%%%%%%%%%%%%%%%%%%%%%%%%%%%%%%%%%%%%%%%%%%%%%
\section{Extension to some integro-differential equation}\label{WEFE}
Let $X$ be a Hilbert space equipped with the norm $\|\,.\,\|_{X}$, and let $A:\mathcal{D}(A)\subset X \rightarrow X$ be a closed, self-adjoint and strictly positive operator on $X$ with dense domain. We introduce the scale of Hilbert spaces $X_{\beta}$, $\beta\in\R$, as follows: for every $\beta \geq 0$, $X_{\beta}={\mathcal D}(A^{\beta})$, with the norm $\|z\|_{\beta}=\|A^{\beta} z\|_{X}$. The space $X_{-\beta}$ is defined by duality with respect to the pivot space $H$ as follows: $X_{-\beta} =X_{\beta}^*$ for $\beta>0$. The operator $A$ can be extended (or restricted) to each $X_\beta$, such that it becomes a bounded operator
$$
A:X_\beta\rarrow X_{\beta-1},\quad \forall\,\beta\in\R \m.
$$
Let a bounded linear operator $B:U\rarrow X_{-\frac{1}{2}}$, where $U$ is another Hilbert space which will be identified with its dual.

We consider the following integro-differential equation
\begin{equation}\label{WEFE1}
\left\{\begin{array}{ll}
\mathrm{\textbf{D}}_{t}^{\alpha}u(t)+g_{\alpha}*Au(t)+BB^{*}u(t)=0,&t>0,
\\
u(0)=u^{0}.&
\end{array}\right.
\end{equation}

We set the Hilbert space $\mathcal{H}=X_{\frac{1}{2}}\times X$ and we consider the unbounded operator $\mathcal{A}:\mathcal{D}{(\mathcal{A})}\longrightarrow\mathcal{H}$ defined by
$$
\mathcal{A}=\left(\begin{array}{cc}
0&I
\\
-A&-BB^{*}
\end{array}\right),
$$
where $\mathcal{D}(\mathcal{A})=\{(v,u)\in\mathcal{H}:\; u\in X_{\frac{1}{2}},\; Av+BB^{*}u\in X\}$. It is well known (see \cite{AHR,AN}) that $\mathcal{A}$ is a generator of a $C_{0}$-semigroup of contractions on $\mathcal{H}$.
\begin{defi}
A function $u\in\mathcal{C}(\R_{+},X)$ such that $g_{\alpha}*u\in\mathcal{C}(\R_{+},X_{\frac{1}{2}})$ is called a strong solution of \eqref{WEFE1} if the couple $\left(\begin{array}{c}g_{\alpha}*u
\\
u
\end{array}\right)\in\mathcal{C}(\R_{+},\mathcal{D}(\mathcal{A}))$, $\left(\begin{array}{c}
g_{1}*u
\\
g_{1-\alpha}*(u-u^{0})
\end{array}\right)\in\mathcal{C}^{1}(\R_{+},\mathcal{H})$ and \eqref{WEFE1} holds on $\R_{+}$ with $u^0 \in X$.
\end{defi}
\begin{defi}\label{WEFE9}
The problem \eqref{WEFE1} is called well-posed if for any $u^{0}\in X_{\frac{1}{2}}$ such that $BB^{*}u^{0}\in X_{\frac{1}{2}}$, there is a unique strong solution $u(t,u^{0})$ of \eqref{WEFE1}, and when $u_{n}^{0}\in\mathcal{D}(\mathcal{A})$ such that $u_{n}^{0}\,\longrightarrow\,0$ as $n\nearrow+\infty$ then $u(t,u_{n}^{0})\,\longrightarrow\,0$ in $X$ and $g_{\alpha}*u(.,u_{n}^{0})(t)\,\longrightarrow\,0$ in $X_{\frac{1}{2}}$ as $n\nearrow+\infty$ in $H$, uniformly on compact intervals.
\end{defi}
\begin{thm} \label{stabint}
Under the above assumptions made on the operator $A$, system \rfb{WEFE1} is well-posed in such away if $u^{0}\in X_{\frac{1}{2}}$ such that $BB^{*}u^{0}\in X$ and we have the following regularity of the solution
$$
u\in\mathcal{C}(\R_{+},X_{\frac{1}{2}}),\quad g_{\alpha}*u\in\mathcal{C}(\R_{+},X_{\frac{1}{2}}),\quad g_{1-\alpha}*(u-u^{0})\in\mathcal{C}^{1}(\R_{+},X)
$$
If in addition, the following properties hold,
\begin{equation}\label{WEFE6}
i\R\subset\rho(\mathcal{A}),\quad\text{and}\quad\limsup_{\mu\in\R, |\mu| \rightarrow + \infty}\|(i\mu I-\mathcal{A})^{-1}\|_{\mathcal{L}(H)}<+\infty.
\end{equation}
Then for some constant $C>0$ and for any data $u^{0}\in X_\half$, the solution $u(t)$ of \eqref{WEFE1} satisfies the following asymptotic estimates
\begin{equation}\label{WEFE7}
\|u(t)\|_{X}\leq\frac{C}{1+t^{\alpha}}\|u^{0}\|_{X},\quad\forall\,t\geq0,
\end{equation}
and
\begin{equation}\label{WEFE8}
\|g_{\alpha} * u (t)\|_{X_{\frac{1}{2}}}\leq\frac{C}{1+t^{\alpha}}\|u^{0}\|_{X},\quad\forall\,t\geq0.
\end{equation}
\end{thm}
\begin{rem}
We note that in the case where $B \in \mathcal{L}(U,X)$ and if $u^0 \in X_\half$ we have immediately that $BB^* u^0 \in X$. 
\end{rem}
\begin{proof}
Let's consider the following equation
\begin{equation}\label{WEFE2}
\left\{\begin{array}{ll}
\mathrm{\textbf{D}}_{t}^{\alpha}U(t)=\mathcal{A}U(t),&t\geq 0,
\\
U(0)=U^{0},&
\end{array}\right.
\end{equation}
where we have denoted by
$$
U(t)=\left(\begin{array}{c}
v(t)
\\
u(t)
\end{array}\right)\qquad\text{and}\qquad U^{0}=\left(\begin{array}{c}
v^{0}
\\
u^{0}
\end{array}\right).
$$
Since the operator $\mathcal{A}$ is m-dissipative (see \cite{AHR}) then according to Theorem \ref{WPEFE17}, system \eqref{WEFE2} is well-posed as given by Definition \ref{WPFE17}. In the other hand, according to section \ref{WPFE} system \eqref{WEFE2} is equivalent to the following integral equation
\begin{equation}\label{WEFE5}
U(t)=U^{0}+g_{\alpha}*\mathcal{A}U(t).
\end{equation}
Equation \eqref{WEFE5} can be also writing as follow:
\begin{equation*}
\left(\begin{array}{c}
v(t)
\\
u(t)
\end{array}\right)=\left(\begin{array}{c}
v^{0}
\\
u^{0}
\end{array}\right)+g_{\alpha}* \left[\left(\begin{array}{cc}
0&I
\\
-A&-BB^{*}
\end{array}\right)\left(\begin{array}{c}
v(t)
\\
u(t)
\end{array}\right) \right].
\end{equation*}
Equivalently, we have
\begin{equation}\label{WEFE3}
u(t)= u^{0}-g_{2\alpha}*Au(t)-g_{\alpha}*Av^{0}-g_{\alpha}*BB^{*}u(t),
\end{equation}
where we have used the semigroup property \eqref{WPFE12}. By taking $v^{0}=0$ and $u^{0}\in X_{\frac{1}{2}}$ such that $BB^{*}u^{0}\in X$ in \eqref{WEFE3}, we obtain
\begin{equation}\label{WEFE4}
u(t)= u^{0}-g_{2\alpha}*Au(t)-g_{\alpha}*BB^{*}u(t).
\end{equation}
Since in this case $U(t)$ is given by the couple $\left(\begin{array}{c}
g_{\alpha}*u(t)
\\
u(t)
\end{array}\right)$ where $u(t)$ is the solution of the system \eqref{WEFE4} and the problem \eqref{WEFE2} is well-posed then by Definition \ref{WPFE17}, $U\in\mathcal{C}(\R,\mathcal{D}(\mathcal{A}))$ and $g_{1-\alpha}*U\in\mathcal{C}(\R,\mathcal{D}(\mathcal{A}))$ then this imply that $u\in\mathcal{C}(\R_{+},X_{\frac{1}{2}})$, $g_{\alpha}*u\in\mathcal{C}(\R_{+},X_{\frac{1}{2}})$ and $g_{1-\alpha}*(u-u^{0})\in\mathcal{C}^{1}(\R_{+},X)$ then by applying the operator $\mathrm{\textbf{D}}_{t}^{\alpha}$ on both sides of\eqref{WEFE4} we find that \eqref{WEFE4} is equivalent to
$$
\mathrm{\textbf{D}}_{t}^{\alpha}u(t)+g_{\alpha}*Au(t)+BB^{*}u(t)=0,\qquad\forall\,t>0,
$$
where we have used here the semigroup property \eqref{WPFE12} and \eqref{WPFE7}. Now we have proved that system \eqref{WEFE2} with $U^{0}=\left(\begin{array}{c}
0
\\
u^{0}
\end{array}\right)$ with $u^{0}\in X_{\frac{1}{2}}$ such that $BB^{*}u^{0}\in X$ is equivalent to the  equation \eqref{WEFE4}. Since \eqref{WEFE2} is well-posed in the sens of Definition \ref{WPFE17} then problem \eqref{WEFE1} is also well-posed in the sens of Definition \ref{WEFE9} and the regularities of the solution $u(t)$ of \eqref{WEFE1} given by the theorem hold.

Now if assumptions \eqref{WEFE6} hold, then according to Theorem \ref{PSFE10} the solution $U$ of \eqref{WEFE2} satisfies the following estimation
$$
\|U(t)\|_{\mathcal{H}}\leq\frac{C}{1+t^{\alpha}}\|U^{0}\|_{\mathcal{H}},\quad\forall\,U^{0}\in\mathcal{H},\;\forall\,t\geq0,
$$
for some constant $C>0$, then we follow
$$
\|u(t)\|_{X}+\|g_{\alpha}*u(t)\|_{X_{\frac{1}{2}}}\leq\frac{C}{1+t^{\alpha}}\|u^{0}\|_{X},\quad\forall\,u^{0}\in X_\half,\;\forall\,t\geq0.
$$
This implies in particular the estimates \eqref{WEFE7} and \eqref{WEFE8} and completes the proof.
\end{proof}

\begin{exam}
We consider the following integro-differential equation:
\begin{equation}\label{examin}
\left\{\begin{array}{ll}
\mathrm{\textbf{D}}_{t}^{\alpha}u(t,x) - g_{\alpha}*\Delta u(t,x)+ a(x) u(t,x)=0,&t>0, x\in \Omega, \\
u = 0, \, (0,+\infty) \times\partial\Omega,
\\
u(0,\cdot)=u^{0} \in H^1_0 (\Omega),&
\end{array}\right.
\end{equation}
where $\Omega$ is a smooth bounded open domain of $\R^n$ and $a \in L^\infty (\Omega)$ is non-identically zero and positive function. 

\medskip

By a direct implication of Theorem \ref{stabint} we obtain a polynomial stability result for \rfb{examin} under a geometric condition G.C.C (see \cite{leb} for more details) on the support of $a$.
\end{exam}
%%%%%%%%%%%%%%%%%%%%%%%%%%%%%%%%%%%%%%%%%%%%%%%%%%%%%%%%%%%%%%%%%%%%%%%%%%%%%%%%%%%%%%%%%%%%%%%%%%%%%%%%%%%%%%%%%%%%%%%%%%%%%%%%%%%%%%%%%%%%%%%%%%%%%%%%%%%%%%%%%%%%%%%%%%%%%%%%%%%%%%%%%%%%%%%%%%%%%%%%%%%%%%%%%%%%%%%%%%%%%%%%%%%%%%%%%%%%%%%%%%%%%%%%%%SECTION%%%%%%%%%%%%%%%%%%%%%%%%%%%%%%%%%%%%%%%%%%%%%%%%%%%%%%%%%%%%%%%%%%%%%%%%%%%%%%%%%%%%%%%%%%%%%%%%%%%%%%%%%%%%%%%%%%%%%%%%%%%%%%%%%%%%%%%%%%%%%%%%%%%%%%%%%%%%%%%%%%%%%%%%%%%%%%%%%%%%%%%%%%%%%%%%%%%%%%%%%%%%%%%%%%%%%%%%%%%%%%%%%%%%%%%%%%%%%%%%%%%%%%%%%%%%
%\section{Application to the Schr\"odinger  equation}
%%%%%%%%%%%%%%%%%%%%%%%%%%%%%%%%%%%%%%%%%%%%%%%%%%%%%%Bibliography%%%%%%%%%%%%%%%%%%%%%%%%%%%%%%%%%%%%%%%%%%%%%%%%%%%%%%%%

\end{document}